\documentclass{article} 

\usepackage[T1]{fontenc}
\usepackage[utf8]{inputenc}
\usepackage{amsmath}
\usepackage{amsfonts,amssymb,latexsym,epsfig,amsthm, amscd}
\usepackage{color}
\usepackage{latexsym,epsfig}
\usepackage{hyperref}
\definecolor{ultramarine}{rgb}{0.15, 0.25, 0.70}
\hypersetup{colorlinks=true,citecolor=ultramarine,linkcolor=ultramarine,urlcolor=ultramarine} 

\usepackage[top=3cm,bottom=3cm,left=3.5cm,right=3.5cm]{geometry}

\title{Stochastic epidemic model with varying infectivity and waning immunity: the law of large numbers with unbounded infectivity}

\author{
  Rapha{\"e}l Forien\footnote{BioSP, INRAE, France.
    {raphael.forien@inrae.fr}, ORCID: {0000-0002-8901-4921}}
  \and 
  {\'E}tienne Pardoux\footnote{Aix--Marseille Universit{\'e}, CNRS, I2M, Marseille, France. 
  {etienne.pardoux@univ-amu.fr}, ORCID: {0000-0002-2586-4791}}}

\newcommand{\E}        {{ {\mathbb E}}}
\renewcommand{\P}      {{ {\mathbb P}}}

\newcommand{\R}        {{\mathbb R}}

\newcommand{\F}  {{\mathcal F}}

\newcommand\A  {\mathfrak{a}}

\newtheorem{theorem}{{Theorem}}
\newtheorem{proposition}[theorem]{{Proposition}}
\newtheorem{lemma}[theorem]{Lemma}

\newtheorem{remark}[theorem]{Remark}

\begin{document}

\maketitle

\begin{abstract}
	We revisit the large population limit of our epidemic model with infection age dependent infectivity and progressive 
	immunity waning, under the assumption that the supremum in $t$ of the random infectivity function has a finite expectation, while the previous proofs assumed that this supremum admits a deterministic upper bound.
	
	~~
	
	\textbf{AMS: } 60F17, 60G55, 92D30.
\end{abstract}

\section{Introduction}

In a series of papers which appeared almost one hundred years ago, see \cite{KK1}, \cite{KK2}, \cite{KK3}, 
Kermack and McKendrick introduced first a SIR model with infection-age dependent infectivity, and then combined 
this with progressive loss of immunity.

In \cite{FPP1} and \cite{FPP2}, the authors have shown that the 1927 SIR model of \cite{KK1} is the law of large numbers limit,
as the size of the population tends to infinity, of 
a stochastic model where the infection-age dependent infectivity of infectious individuals is random, 
an independent copy of that random function being attached to each individual.

Later in \cite{FPPZ}, it has been shown that the 1932-33  model of \cite{KK2}, \cite{KK3} is a particular case of the law of large numbers limit
of a stochastic model where both the infection age dependent infectivity and the progressive waning of immunity are random, and i.i.d. among the
various individuals and successive infections.

More precisely, in the model introduced in \cite{FPPZ}, at each new infection, a new independent copy of a pair of random functions $ (\lambda, \gamma) $ is drawn, where $ \lambda $ is the infectivity function and $\gamma$ is the suceptibility function.
Letting $(\lambda_{k,i}, \gamma_{k,i})$ denote the pair drawn for the $i$-th infection of individual $k$, the infectivity of individual $k$ at time $t$ after their $i$-th infection and before their next infection is given by $\lambda_{k,i}(\A(t))$ and the susceptibility of the same individual $k$ is given by $\gamma_{k,i}(\A(t))$, where $\A(t)$ denotes the time elapsed at time $t$ since their $i$-th infection.
The main result of \cite{FPPZ} is that, as $N\to\infty$, the pair of processes given by the renormalized total force of infection in the population, and the average susceptibility, converges in probability, locally uniformly in $t$, to the unique solution of a system of two integral equations, which can be seen as a generalization of the Kermack-McKendrick system of equations.

In all those works, the random infectivity function was assumed to admit an almost sure deterministic upper bound. In the present note, we show that the main results
of the above works remain true if we assume only that the supremum in $t$ of the random infectivity admits a finite expectation. We treat both the case of the SIR model, and the case a progressive loss of immunity.

{\bf Notations} $\mathbf{D} = D(\R_+;\R)$ stands for the Skorohod space of c\`adl`ag real-valued functions, endowed with the usual topology and $\mathbf{D}^2 = D(\R_+;\R^2)$ is the space of c\`adl`ag functions taking values in $ \R^2 $.

\section{Statement of the result} \label{sec:statement}

We consider a population of size $N$. To each individual $1\le k\le N$, we associate
its initial status which is summarized by the i.i.d. $\mathbf{D}^2$-valued random variables $\lbrace (\lambda_{k,0},\gamma_{k,0}), 1 \leq k \leq N \rbrace$, where $ t \mapsto \lambda_{k,0}(t) $ (resp. $ t \mapsto \gamma_{k,0}(t) $) represents the infectivity (resp. susceptibility) of individual $ k $ up to their next infection.
A classical example of initial conditions, compatible with our model, is the following. Let $\bar{I}(0)$ (resp. $\bar{S}(0)$) denote the proportion of
initially infected (resp. of initially susceptible) individuals. We assume that $\bar{I}(0)+\bar{S}(0)=1$. 
Then, given $\{X_k,\ 1\le k\le N\}$ a family of i.i.d. Bernoulli random variables with parameter  $\bar{I}(0)$ and
$\{(\lambda^0_k,\gamma^0_k), \ 1\le k\le N\}$ i.i.d. $\mathbf{D}^2$-valued random functions, we set
\begin{align}\label{init}
\lambda_{k,0}(t)=X_k\lambda^0_k(t),\quad \gamma_{k,0}(t)=1-X_k+X_k\gamma^0_k(t) \quad t \geq 0\,.
\end{align}
Also let $\{(\lambda_{k,i}, \gamma_{k,i}), 1\le k\le N, i\ge1\}$ be a family of i.i.d. $ \mathbf{D}^2 $-valued random variables, independent from the previous family.
The pair $ (\lambda_{i,k}, \gamma_{i,k}) $ then gives the infectivity and susceptibility of individual $ k $ after their $ i $-th infection and up to their next infection.
Let $ (\lambda_{0}, \gamma_0) $ (resp. $ (\lambda, \gamma) $) denote a generic $ \mathbf{D}^2 $-valued random variable with the same distribution as $ (\lambda_{k,0}, \gamma_{k,0}) $ (resp. as $ (\lambda_{k,1}, \gamma_{k,1}) $).
Assumptions on the law of $ (\lambda_{k,i}, \gamma_{k,i}) $ are thus stated below as assumptions on $ (\lambda_0, \gamma_0) $ and $ (\lambda, \gamma) $.

The number of (re)infections of individual $k$ on the time interval $(0,t]$ is denoted by $B^N_k(t)$ and will be defined as the solution to a stochastic equation below.
Defining the age of (the most recent) infection of individual $ k $,
\[ \A^N_k(t)=t-(\sup\{s\in(0,t], B^N_k(s)=B^N_k(s^-)+1\}\vee0)\,,\]
the infectivity and susceptibility of individual $k$ at time $t$ is given by the pair
\begin{equation*}
	\Big( \lambda_{k,B^N_k(t)}(\A^N_k(t)), \, \gamma_{k,B^N_k(t)}(\A^N_k(t)) \Big).
\end{equation*}

We assume that $0\le \lambda_0(t), \lambda(t)$ and $0\le \gamma_0(t),\gamma(t)\le 1$ almost surely, and that
\begin{equation} \label{disjoint_support}
    \inf\{t \geq 0,\ \gamma(t)>0\}\ge\sup\{t \geq 0,\ \lambda(t)>0\}, \quad \text{ almost surely,}
\end{equation}
and we set $\lambda(t)=\lambda_0(t)=\gamma(t)=\gamma_0(t)=0$ for $t<0$.
Setting $\lambda^\ast:=\sup_{t\ge0}\lambda(t)$ and $\lambda_{0}^\ast:=\sup_{t\ge0}\lambda_0(t)$, we assume moreover that
\begin{equation} \label{assumption_1st_moment}
	\E[\lambda^\ast]<\infty \quad\text{and}\quad \E[\lambda^\ast_0]<\infty.
\end{equation}

The renormalized total force of infection and the renormalized total susceptibility are given by
\begin{equation}\label{FNSN}
\overline{\mathfrak{F}}^N(t)=\frac{1}{N}\sum_{k=1}^N\lambda_{k,B^N_k(t)}(\A^N_k(t)),\quad\text{ and } \quad
 \overline{\mathfrak{S}}^N(t)=\frac{1}{N}\sum_{k=1}^N\gamma_{k,B^N_k(t)}(\A^N_k(t))\,.
 \end{equation}
Given $\{Q_k,\ 1\le k\le N\}$ i.i.d. standard Poisson random measures on $\R_+^2$, $(B^N_k(\cdot), 1 \leq k \leq N)$ is defined as the solution of the system of stochastic equations
\begin{equation*}
    \left\lbrace
    \begin{aligned}
        B^N_k(t)&=\int_0^t\int_0^\infty{\bf1}_{u\le\Upsilon^N_k(s^-)}Q_k(ds,du),\\
        \Upsilon^N_k(s)&=\gamma_{k,B^N_k(s)}(\A^N_k(s))\overline{\mathfrak{F}}^N(s)\,,
    \end{aligned}
    \right. \quad 1 \leq k \leq N.
\end{equation*}
The above process almost surely does not explode in finite time, since consecutive jumps (i.e. infections) are at least separated by an interval during which a new copy of $\gamma_{k,i}$ is zero, whose length is non zero with positive probability.

We will use below the following notations: 
\begin{align*}
	\overline{\lambda}_0(t)=\E[\lambda_0(t)], && \overline{\lambda}(t)=\E[\lambda(t)], && \overline{\lambda}^\ast = \E[\lambda^\ast] \vee \E[\lambda_0^\ast],
\end{align*}
where the latter is an upper bound of both $\sup_{t\ge0}\overline{\lambda}_0(t)$ and $\sup_{t\ge0}\overline{\lambda}(t)$.

 Our main result is the following law of large numbers.

\begin{theorem}\label{LLNVIWI} Under the above assumptions,
as $N\to\infty$, $(\overline{\mathfrak{S}}^N(\cdot),\overline{{\mathfrak{F}}}^N(\cdot))$ converges in probability, locally uniformly, to $ (\overline{\mathfrak{S}}(\cdot), \overline{\mathfrak{F}}(\cdot)) $, defined as the unique solution of 
\begin{equation}\label{eq:limit}
    \left\lbrace
\begin{aligned}
\overline{\mathfrak{S}}(t)&=\E\left[\gamma_0(t)\exp\left(-\int_0^t\gamma_0(r)\overline{\mathfrak{F}}(r)dr\right)\right]
\\
&\quad+\int_0^t \E\left[\gamma(t-s)\exp\left(-\int_s^t\gamma(r-s)\overline{\mathfrak{F}}(r)dr\right)\right]
\overline{\mathfrak{S}}(s) \overline{\mathfrak{F}}(s)ds ,
\\
\overline{\mathfrak{F}}(t)&=\overline{\lambda}_0(t)+\int_0^t\overline{\lambda}(t-s)\overline{\mathfrak{S}}(s) \overline{\mathfrak{F}}(s)ds\,.
\end{aligned}
    \right.
\end{equation}
\end{theorem}

In the situation described in \eqref{init} above, we have
\begin{align*}
\E\left[\gamma_0(t)\exp\left(-\int_0^t\gamma_0(r)\overline{\mathfrak{F}}(r)dr\right)\right]&=  \bar{S}(0)\exp\left(-\int_0^t\overline{\mathfrak{F}}(r)dr\right) \\ &\quad+\bar{I}(0)
\E\left[\gamma^0(t)\exp\left(-\int_0^t\gamma^0(r)\overline{\mathfrak{F}}(r)dr\right)\right] \\
\overline{\lambda}_0(t)&=\bar{I}(0)\overline{\lambda^0}(t)\,,
\end{align*}
where $\overline{\lambda^0}(t)=\E[\lambda^0_k(t)]$.

From this result, one can easily deduce (see \cite{FPPZ}) expressions for the proportions of infected and uninfected individuals.
For any $1\le k\le N$, $i\ge0$, let $\eta_0=\sup\{t>0,\ \lambda_0(t)>0\}$ and $\eta_{k,i}=\sup\{t>0,\ \lambda_{k,i}(t)>0\}$, $F_0^c(t)=\P(\eta_0>t)$, $F^c(t)=\P(\eta>t)$.
The proportions of infected and uninfected individuals are
\begin{align*}
    \overline{I}^N(t)=\frac{1}{N}\sum_{k=1}^N{\bf1}_{\{\A^N_k(t)<\eta_{k,B^N_k(t)}\}}, && \text{ and } && \overline{U}^N(t)=\frac{1}{N}\sum_{k=1}^N{\bf1}_{\{\A^N_k(t)\ge\eta_{k,B^N_k(t)}\}},
\end{align*}
respectively.
We have the following result.
\begin{proposition}
As $N\to\infty$, $(\overline{U}^N(t),\overline{I}^N(t))\to(\overline{U}(t),\overline{I}(t))$ in probability, locally uniformly in $t$, where
\begin{align*}
\overline{U}(t)&=\E\left[{\bf1}_{t\ge\eta_0}\exp\left(-\int_0^t\gamma_0(r)\overline{\mathfrak{F}}(r)dr\right)\right]\\
&\quad+\int_0^t\E\left[{\bf1}_{t-s\ge\eta}\exp\left(-\int_s^t\gamma(r-s)\overline{\mathfrak{F}}(r)dr\right)\right]
\overline{\mathfrak{S}}(s)\overline{\mathfrak{F}}(s)ds,\\
\overline{I}(t)&=\overline{I}(0)F_0^c(t)+\int_0^tF^c(t-s)\overline{\mathfrak{S}}(s)\overline{\mathfrak{F}}(s)ds\,.
\end{align*}
\end{proposition}

Theorem~\ref{LLNVIWI} generalises the main result of \cite{FPPZ}, in which it was assumed that there exists a constant $ \lambda^{**} > 0 $ such that $ \lambda^* \leq \lambda^{**} $ and $ \lambda_0^* \leq \lambda^{**} $, almost surely.
The proof presented here only requires \eqref{assumption_1st_moment}, i.e. that $ \lambda^* $ and $ \lambda^*_0 $ admit a finite first moment.
Before proving Theorem~\ref{LLNVIWI} in full generality, we treat the case of the SIR model in the next Section, in which individuals may only be infected once, as the proof is much simpler in this case.
The proof of the general result is presented in Section~\ref{sec:proof}.

\section{The particular case of the SIR model}

In this section,  we assume that individuals who have been infected at some time never lose their immunity.
Those who have never been infected are fully susceptible. This means that we are in the situation described by formulas \eqref{init} and that $\gamma^0\equiv \gamma \equiv0$. 
In particular
\begin{align*}
\begin{cases} \gamma_0(t)=1,&\text{for all $t\ge0$, with probability $\bar{S}(0)$,}\\
\gamma_0(t)=0,&\text{for all $t\ge0$, with probability $\bar{I}(0)$}\,,
\end{cases}
\end{align*}
where $\bar{S}(0)+\bar{I}(0)=1$. In the present case, the total susceptibility in the population is the number of susceptible individuals, and the renormalized total susceptibility
equals the proportion $\bar{S}^N(t)=N^{-1}S^N(t)$ of susceptible individuals, which converges to $\bar{S}(t)$ as $N\to\infty$, $\bar{S}^N$ (resp. $\bar{S}$) replacing here the quantity $\overline{\mathfrak{S}}^N$ (resp. $\overline{\mathfrak{S}}$). Moreover, the system \eqref{eq:limit} becomes here 
\begin{equation} \label{limit_SIR}
	\left\lbrace
	\begin{aligned}
		\bar{S}(t)&=\bar{S}(0)\exp\left(-\int_0^t\overline{\mathfrak{F}}(s)ds\right),\\
		\overline{\mathfrak{F}}(t)&=\bar{I}(0)\overline{\lambda^0}(t)+\int_0^t\overline{\lambda}(t-s)\bar{S}(s) \overline{\mathfrak{F}}(s)ds\,.
	\end{aligned}
	\right.
\end{equation}
We note that in fact $\bar{S}(t)$ solves the ODE
\[\bar{S}(t)=\bar{S}(0)-\int_0^t\bar{S}(s)\overline{\mathfrak{F}}(s)ds\,.\]
We can also rewrite the second equation as 
\begin{equation}\label{SIRlimeq}
 \overline{\mathfrak{F}}(t)=\bar{I}(0)\overline{\lambda^0}(t)+\bar{S}(0)\int_0^t\overline{\lambda}(t-s) \overline{\mathfrak{F}}(s)\exp\left(-\int_0^s\overline{\mathfrak{F}}(r)dr\right)ds\,.
 \end{equation}
Clearly, if $\overline{\mathfrak{F}}$ solves this last equation, then
\begin{align*}
\overline{\mathfrak{F}}(t)\le \bar{I}(0)\overline{\lambda}^\ast+\bar{S}(0)\overline{\lambda}^\ast\int_0^t\overline{\mathfrak{F}}(s)\exp\left(-\int_0^s\overline{\mathfrak{F}}(r)dr\right)ds
\le\overline{\lambda}^\ast\,.
\end{align*} 
With this a priori bound, it is easy to show that the last equation has at most one solution. Existence follows by the same Picard iteration argument as for ODEs.

The stochastic model can be described as follows. 
Recall that $ (X_k, k \geq 1) $ is an i.i.d. sequence of Bernoulli random variables with parameter $ \bar{I}(0) $, $ (\lambda_k^0, k \geq 1) $ is an i.i.d. sequence of $ \mathbf{D} $-valued random variables, and set $ \lambda_k = \lambda_{k,1} $, where we recall that $ (\lambda_{k,1}, k \geq 1) $ is another i.i.d. sequence of $ \mathbf{D} $-valued random variables.
Recall further that all the above sequences are mutually independent and that we consider $\{Q_k,\ k \geq 1\}$, an i.i.d. family of standard Poisson random measures on $\R^2_+$ (also independent of the previous random variables).
The individual $k$ is initially susceptible (resp. infected) if $X_k=0$ (resp. $X_k=1$). 
In addition, individual $k$ is susceptible at time $t$ if $X_k=0$ and $B^N_k(t)=0$; otherwise it is either infected or removed.
Then the above model reduces to
\begin{align*}
 B^N_k(t)&=(1-X_k)\int_0^t\int_0^\infty{\bf1}_{B^N_k(s^-)=0}{\bf1}_{u\le \overline{\mathfrak{F}}^N(s^-)}Q_k(ds,du), \quad 1 \leq k \leq N,\\  
 \overline{\mathfrak{F}}^N(t)&=N^{-1}\sum_{k=1}^N \left\lbrace X_k\lambda^0_k(t)+(1-X_k)\lambda_k(\A^N_k(t)) \right\rbrace,
 \end{align*}
where
\begin{align*}
	\tau^N_k:&=\inf\{t\ge0,\ B^N_k(t)=1\}, \quad \text{ and } \quad \A^N_k(t)=t-\tau^N_k.
\end{align*}
Furthermore,
\begin{align*}
 \bar{S}^N(t)&=N^{-1}\sum_{k=1}^N(1-X_k){\bf1}_{B^N_k(t)=0}\,.
 \end{align*}
We now define, for $k \geq 1$,
\begin{align} \label{Bk_SIR}
B_k(t)=(1-X_k)\int_0^t\int_0^\infty {\bf1}_{B_k(s^-)=0}{\bf1}_{u\le \overline{\mathfrak{F}}(s)}Q_k(ds,du),
\end{align}
where
\begin{equation*}
	\tau_k:=\inf\{t\ge0,\ B_k(t)=1\}, \quad \text{ and } \quad \A_k(t):=t-\tau_k.
\end{equation*}

The aim of this section is to show that $ (\overline{\mathfrak{F}}^N, \overline{S}^N)$ converges in probability (locally uniformly in time) to $ (\overline{\mathfrak{F}}, \overline{S}) $, given as the solution to \eqref{limit_SIR}.
We first show the following.
\begin{lemma}
For any $k \geq 1$ and $ t \geq 0 $,
\begin{equation*}
	\E[X_k\lambda^0_k(t)+(1-X_k)\lambda_k(\A_k(t))]=\overline{\mathfrak{F}}(t).
\end{equation*}
\end{lemma}

\begin{proof}
	 For any $k \geq 1$, 
\[ (1-X_k)\lambda_k(\A_k(t)) = (1-X_k)\int_0^t\int_0^\infty {\bf1}_{B_k(s^-)=0}\lambda_k(t-s){\bf1}_{u\le \overline{\mathfrak{F}}(s)}Q_k(ds,du),\]
whose expectation equals
\begin{align*}
\bar{S}(0)\int_0^t\overline{\lambda}(t-s)\overline{\mathfrak{F}}(s)\exp\left(-\int_0^s\overline{\mathfrak{F}}(r)dr\right)ds\,.
\end{align*}
The result follows by comparing the last formula with \eqref{SIRlimeq}, and noting that $ \E[X_k \lambda^0_k(t)] = \bar{I}(0) \overline{\lambda^0}(t) $.
\end{proof}

Moreover $\{X_k\lambda^0_k(t)+(1-X_k)\lambda_k(\A_k(t)),\ k\ge1\}$ is an i.i.d. sequence of random elements of $D$. Consequently, also noting that
$\P(X_k=0,B_k(t)=0)=\bar{S}(0)\exp\left(-\int_0^t\overline{\mathfrak{F}}(s)ds\right)$, 
 it follows from the law of large numbers in $D$ (see 
\cite{Rao}) that, as $ N \to \infty $,
\begin{align*}
\left( t \mapsto N^{-1}\sum_{k=1}^N (1-X_k){\bf1}_{B_k(t)=0} \right)  &\to\bar{S}(\cdot),\\
\left( t \mapsto N^{-1}\sum_{k=1}^N\{X_k\lambda^0_k(t)+(1-X_k)\lambda_k(\A_k(t))\} \right)&\to \overline{\mathfrak{F}}(\cdot),
\end{align*}
in $ \mathbf{D} $ almost surely.
Hence the law of large numbers in the SIR case will follow if we can replace above $B_k(t)$ by $B^N_k(t)$ and $\A_k(t)$ by $\A^N_k(t)$. This is a consequence of the follwing Lemma.
\begin{lemma}
As $N\to\infty$, for any $T>0$,
    \begin{align*}
        \frac{1}{N}\E\left[\sum_{k=1}^N\sup_{0\le t\le T}|B^N_k(t)-B_k(t)|\right]\to0\,.
    \end{align*}
\end{lemma}

\begin{proof}
By \eqref{Bk_SIR},
\begin{align*}
\sup_{0\le s\le t}|B^N_k(s)-B_k(s)|&\le \int_0^t\int_{\overline{\mathfrak{F}}^N(s^-)\wedge\overline{\mathfrak{F}}(s)}^{\overline{\mathfrak{F}}^N(s^-)\vee\overline{\mathfrak{F}}(s)} Q_k(ds,du).
\end{align*}
Hence
\begin{align*}
\E\left[\sup_{0\le s \le t}|B^N_k(s)-B_k(s)|\right]
	&\le\E\left[\int_0^t|\overline{\mathfrak{F}}^N(s)-\overline{\mathfrak{F}}(s)|ds\right]\\
&\le N^{-1}\sum_{j=1}^N \E\left[\left|\int_0^t(1-X_j)[\lambda_j(s-\tau^N_j)-\lambda_j(s-\tau_j)]ds\right|\right]\\
&\quad+\E\Bigg[\int_0^t\Big| \overline{\mathfrak{F}}(s)-N^{-1}\sum_{j=1}^N\{X_j\lambda^0_j(s)+(1-X_j)\lambda_j(\A_j(s))\}\Big|ds\Bigg].
\end{align*}
It follows from the law of large numbers for i.i.d. sequences (where the convergences holds in $L^1(\Omega)$, see e.g. Exercise 24 page 59 in \cite{Kal})
that the last term on the right tends to $0$ locally uniformly in $t$, as $N\to\infty$. Moreover the first term on the right is bounded by
\begin{align*}
 N^{-1}\sum_{j=1}^N\E\left[\int_0^t|\lambda_j(s-\tau^N_j)-\lambda_j(s-\tau_j)|ds\right]
&\le N^{-1}\sum_{j=1}^N\E\left[\int_0^t\lambda_j^\ast{\bf1}_{\lbrace\tau^N_j\wedge s\not=\tau_j\wedge s\rbrace}ds \right]\\
&\le N^{-1}\sum_{j=1}^N\E\left[\int_0^t\lambda_j^\ast\sup_{0\le r\le s}|B^N_j(r)-B_j(r)|ds\right]\\
&\le \overline{\lambda}^\ast \int_0^t\E\left[\sup_{0\le r\le s}|B^N_k(r)-B_k(r)|\right]ds,
\end{align*}
where we have used the independence between the processes $\lambda_j$ and $(B^N_j,B_j)$, and the exchangeability of the sequence $(B^N_j,B_j)$. The result follows from the above estimates and Gronwall's lemma.

\end{proof}

This concludes the proof in the case of the SIR model, and we now return to the general model of Section~\ref{sec:statement}.

\section{Existence and uniqueness of the deterministic limiting system}
We rewrite \eqref{eq:limit} as the system
\begin{equation}\label{eq:limit2}
	\left\lbrace
\begin{aligned}
x(t)&=\E\left[\gamma_0(t)\exp\left(-\int_0^t\gamma_0(r)y(r)dr\right)\right]
\\
&\qquad+\int_0^t \E\left[\gamma(t-s)\exp\left(-\int_s^t\gamma(r-s)y(r) dr\right)\right]x(s) y(s)ds ,
\\
y(t)&=\overline {\lambda}_0(t)+\int_0^t\overline{\lambda}(t-s)y(s) x(s)ds\,.
\end{aligned}
\right.
\end{equation}
We have the following.

\begin{theorem}
The system of integral equations \eqref{eq:limit2} has a unique solution which belongs to $\mathbf{D}^2$, and 
for all $t\ge0$, $0\le x(t)\le 1$ and $0\le y(t)\le \overline{\lambda}^{\ast}_0 \exp(\overline{\lambda}^\ast t)$. Moreover, the solution of this system is continuous whenever $\overline{\lambda}_0(t)$ is continuous, and $\gamma_0(t)$ is continuous in probability.\footnote{In the statement of Theorem 3.1 of \cite{FPPZ}, the stated condition ``$\gamma_0$ has bounded variation and the map $t\mapsto(\E[\gamma_0(t)],\overline{\lambda}_0(t))$ is continuous'' does not imply the continuity. The last inequality in the proof of that theorem is wrong, and counterexamples are easy to find.}
\end{theorem}

 \begin{proof}
 	{\it Step 1 : A priori estimates} Let $(x,y)$ be a solution of the system \eqref{eq:limit2}. We first note that $x(0)\geq 0$ and $y(0)\geq 0$. It is easy to prove that there cannot be any $t>0$ such that $x(t)<0$ or $y(t)< 0$. Hence $x(t)\ge0$ and $y(t)\ge0$ for all $t\ge0$. From the first equation of  \eqref{eq:limit2}, the time derivative of 
\[z(t)\!\!:=\!\!\E\!\left[\!\exp\left(\!\!-\!\!\int_0^t\!\!\gamma_0(r)y(r)dr\!\!\right)\!\!\right]\!\!
+\!\!\int_0^t\!\! \E\!\left[\!\exp\!\!\left(\!\!-\!\!\int_s^t\!\!\gamma(r-s)y(r) dr\!\!\right)\!\!\right]x(s) y(s)ds\]
is zero, which implies that $z(t)=z(0)=1$ for all $t\ge0$. This, combined with the facts that $\gamma_0(t)\le 1$ and $\gamma(t-s)\le1$, implies that $x(t)\le1$.
Plugging this bound in the second equation of \eqref{eq:limit2} and using Gronwall's Lemma, we obtain that $y(t)\le \overline{\lambda}_0^\ast e^{\overline{\lambda}^\ast t}$ for all $t\ge0$.

{\it Step 2 : Uniqueness} Thanks to the above a priori estimates, uniqueness follows by standard arguments.

{\it Step 3 : Existence} Let $F(t;x,y)$ (resp. $G(t;x,y)$) denote the right-hand side of the first (resp. the second) equation
in \eqref{eq:limit2}. Standard computations show that the usual Picard iteration applied to the system
\[ x(t)=F(t;x,y)\wedge 2,\quad y(t)=G(t;x,y)\wedge (2 \overline{\lambda}_0^\ast \exp(\overline{\lambda}^\ast t)) \]
yields the existence of a solution to that system. Moreover, that solution verifies $x(0)=\E[\gamma_0(t)]<2$ and
$y(0)=\overline{\lambda}_0(0)<2\overline{\lambda}_0^\ast $. Since $(x,y)$ is right continuous, 
\[ \tau:=\inf\{t,\ F(t;x,y)\ge 2 \text { or } G(t;x,y)\ge 2 \overline{\lambda}_0^\ast \exp(\overline{\lambda}^\ast t)\}\ > \ 0\,.\]
Since $(x,y)$ solves \eqref{eq:limit2} on $[0,\tau)$, $dz(t)/dt=0$ on $[0,\tau)$, so if $\tau<\infty$, $z(\tau)=1$. 
Hence $x(t)\leq 1$ for all $t\in[0,\tau]$ and consequently $y(t) \leq \overline{\lambda}_0^\ast \exp(\overline{\lambda}^\ast t)$ for all $t\in[0,\tau]$, which contradicts the definition of $\tau$ (using the right-continuity of $(x,y)$ again), hence $\tau=+\infty$, and we have existence of a solution to \eqref{eq:limit2}.

{\it Step 4 : Continuity} Clearly, for the solution to be continuous, it suffices that the two first terms on the right hand sides in \eqref{eq:limit2} are continuous,
which is implied by our additional conditions.

 \end{proof}

\section{Proof of Theorem \ref{LLNVIWI}} \label{sec:proof}

From now on, $(\overline{\mathfrak{S}}(t),\overline{\mathfrak{F}}(t))$ denotes the unique solution of \eqref{eq:limit2}.
The key point of the proof is to show that, for each $ k $, on an event whose probability tends to 1 as $N\to\infty$, $B^N_k(t)$ and $\A^N_k(t)$ coincide with the processes $B_k(t)$ and $\A_k(t)$ defined as the solutions of the following equations, for $k \geq 1$,
\begin{equation*}
    \left\lbrace
    \begin{aligned}
        B_k(t)&=\int_0^t\int_0^\infty{\bf1}_{u\le\Upsilon_k(s^-)}Q_k(ds,du),\\
        \Upsilon_k(s)&=\gamma_{k,B_k(s)}(\A_k(s))\overline{\mathfrak{F}}(s),\\
        \A_k(t)&= t-(\sup\{s\in(0,t], B_k(s)=B_k(s^-)+1\}\vee0)\,.
    \end{aligned}
    \right.
\end{equation*}
We note that, for each $ k \geq 1 $, $(B_k(\cdot),\A_k(\cdot))$ is measurable with respect to the sigma-field generated by $ \lbrace Q_k, (\gamma_{k,i}, i \geq 1) \rbrace $, and thus $ \lbrace (B_k(\cdot), \A_k(\cdot)), k \geq 1 \rbrace $ is an i.i.d. sequence.
Their common law is that of the solution of the following equation
\begin{equation*}
    \left\lbrace
    \begin{aligned}
        B(t)&=\int_0^t\int_0^\infty{\bf1}_{u\le\Upsilon(s^-)}Q(ds,du),\\
        \Upsilon(s)&=\gamma_{B(s)}(\A(s))\overline{\mathfrak{F}}(s),\\
        \A(t)&= t-(\sup\{s\in(0,t], B(s)=B(s^-)+1\}\vee0)\,,
    \end{aligned}
    \right.
\end{equation*}
where $Q$ is a standard Poisson random measure on $\R_+^2$ and $\lbrace \gamma_i, i \geq 0 \rbrace$ is a sequence of random variables distributed as $\lbrace \gamma_{1,i}, i \geq 0 \rbrace$, independent of $Q$.
We further note that existence and uniqueness of the solution to that equation is elementary.

Let us first prove two fundamental  identities.
\begin{proposition}\label{prop:ident}
For all $t\ge0$,
\[(\E[\gamma_{B(t)}(\A(t))],\E[\lambda_{B(t)}(\A(t))]) = (\overline{\mathfrak{S}}(t),\overline{\mathfrak{F}}(t))\,.\]
\end{proposition}

\begin{proof}
Let $\{\tau_i,i\ge1\}$ denote the successive jump times of $B(t)$.
Then
\begin{equation*}
 \gamma_{B(t)}(\A(t))=\gamma_0(t){\bf1}_{B(t)=0}+\sum_{i\ge1}\gamma_i(t-\tau_i){\bf1}_{B(t)=i}.
\end{equation*}
For the first term, we write
\begin{align}
	\E[\gamma_0(t){\bf1}_{B(t)=0}]&=\E[\gamma_0(t)\P(B(t)=0|\gamma_0)] \notag \\
&=\E\left[\gamma_0(t)\exp\left(-\int_0^t\gamma_0(r)\overline{\mathfrak{F}}(r)dr\right)\right]. \label{Egamma0}
\end{align}
Setting $\F_t=\sigma\{(\lambda_i,\gamma_i)_{i\le B(t)}, Q|_{[0,t]\times\R_+} \}$, we note that $\F_{\tau_i}$ and $Q|_{(\tau_i,t]\times\R_+}$  are independent. Hence on the event $\{\tau_i<t\}$,
\begin{equation*}
    \P(B(t)=i|\F_{\tau_i})=\exp\left(-\int_{\tau_i}^t\gamma_i(r-\tau_i)\overline{\mathfrak{F}}(r)dr\right).
\end{equation*}
Denoting by $\mu$ the law of $\gamma$, let us set, for $ 0 \leq s \leq t $,
\begin{equation*}
	g(s,t) = \int_{\mathbf{D}} \gamma(t-s) \exp\left(-\int_{s}^t \gamma(r-s) \overline{\mathfrak{F}}(r)dr \right) \mu(d\gamma).
\end{equation*}
Then
\begin{align}
\E\Bigg[\sum_{i\ge1}\gamma_i(t-\tau_i){\bf1}_{B(t)=i} \Bigg] \notag &=\sum_{i\ge1}\E\left[\gamma(t-\tau_i)\exp\left(-\int_{\tau_i}^t\gamma(r-\tau_i)\overline{\mathfrak{F}}(r)dr\right)\right]\\
    &=\sum_{i\ge1}\E\left[g(\tau_i,t)\right] \notag \\
    &=\E\left[\int_0^tg(s,t) dB(s) \right] \notag \\
    &=\int_0^t g(s,t) \E[\gamma_{B(s)}(\A(s))]\overline{\mathfrak{F}}(s)ds. \label{Egamma_positive}
\end{align}
Combining this with \eqref{Egamma0}, we obtain
\begin{multline*}
\E[\gamma_{B(t)}(\A(t))]=\E\left[\gamma_0(t)\exp\left(-\int_0^t\gamma_0(r)\overline{\mathfrak{F}}(r)dr\right)\right]\\
+\int_0^t\E\left[\gamma(t-s)\exp\left(-\int_{s}^t\gamma(r-s)\overline{\mathfrak{F}}(r)dr\right)\right] \E[\gamma_{B(s)}(\A(s))]\overline{\mathfrak{F}}(s)ds.
\end{multline*}
This proves that $\E[\gamma_{B(t)}(\A(t))]=\overline{\mathfrak{S}}(t)$, as a consequence of the uniqueness of a solution to that 
linear integral equation (which follows from the fact that $\overline{\mathfrak{F}}(t)\le\overline{\lambda}_0^\ast \exp(\overline{\lambda}^\ast t)$ for all $t\ge0$). 
%
%
Let us now show that $ \E[\lambda_{B(t)}(\A(t))] = \overline{\mathfrak{F}}(t) $.
For this, we write
\begin{align*}
\lambda_{B(t)}(\A(t))&=\lambda_0(t){\bf1}_{B(t)=0}+\sum_{i\ge1}\lambda_i(t-\tau_i){\bf1}_{B(t)=i}\\
&=\lambda_0(t)+\sum_{i\ge1}\lambda_i(t-\tau_i)
\end{align*}
Indeed, if $B(t)<i$, $t-\tau_i<0$ and $\lambda_i(t-\tau_i)=0$, while if $B(t)\ge i+1$, $\lambda_i(t-\tau_i)=0$ since then necessarily there exists $s\le t$ such that $\gamma_i(s-\tau_i)>0$, see \eqref{disjoint_support}.
 Then, proceeding as in \eqref{Egamma_positive},
\begin{align*}
\E[\lambda_{B(t)}(\A(t))]&=\E[\lambda_0(t)]+\E\Big[\sum_{i\ge1}\lambda_i(t-\tau_i)\Big]\\
&=\overline{\lambda}_0(t)+\int_0^t\overline{\lambda}(t-s)\E[\gamma_{B(s)}(\A(s))]\overline{\mathfrak{F}}(s)ds.
\end{align*}
The second part of the result then follows from the fact that $\E[\gamma_{B(t)}(\A(t))]=\overline{\mathfrak{S}}(t)$.
\end{proof}


\begin{remark} It follows from the last result that
    \begin{equation*}
        \Upsilon(t)=\gamma_{B(t)}(\A(t))\E\left[\lambda_{B(t)}(\A(t))\right].
    \end{equation*}
If we take that definition of the process $\Upsilon$, we see that the coefficient in the SDE for $B(t)$ is a function of the law of the solution. In this light, the equation for $B(t)$ is a McKean-Vlasov type SDE, which is harder to solve than the one we have used above. Such a McKean-Vlasov SDE was solved in the previous work \cite{FPPZ}.
\end{remark}
In order to prove Theorem \ref{LLNVIWI}, it remains to show that the error resulting from replacing $\{(B^N_k,\A^N_k),\ 1\le k\le N\}$ by $\{(B_k,\A_k),\ 1\le k\le N\}$ in the formulas \eqref{FNSN} tends to zero as $N\to\infty$.
There is a difficulty in establishing such a result under our weak assumption on the random function $\lambda$ (i.e. the fact that $\sup_{t\ge0}\lambda(t)$
is only supposed to be integrable). That difficulty is related to the fact that, under our assumptions on $\lambda$, we have no deterministic uniform bound on the mean infectivity
$\overline{\mathfrak{F}}^N(t)$. For that reason, we introduce the following truncated model.
We choose an arbitrary $T>0$, and prove Theorem \ref{LLNVIWI} on $[0,T]$. 
Let us set $M:=1+\overline{\lambda}^{\ast}_0 \exp(\overline{\lambda}^\ast T) > \sup_{t \in [0,T]} \overline{\mathfrak{F}}(t)$, and, given the same sequence $\{Q_k,\ 1\le k\le N\}$ of i.i.d. standard Poisson random measures on $\R_+^2$ as above, for each $N\ge 1$, let $\lbrace(B^{N,M}_k,\A^{N,M}_k), 1 \leq k \leq N \rbrace$ be the solution to the system of stochastic equations
\begin{equation*}
    \left\lbrace
    \begin{aligned}
        B^{N,M}_k(t)&=\int_0^t\int_0^\infty{\bf1}_{u\le\Upsilon^{N,M}_k(s^-)}Q_k(ds,du),\\
        \Upsilon^{N,M}_k(s)&=\gamma_{k,B^{N,M}_k(s)}(\A^{N,M}_k(s))\overline{\mathfrak{F}}^{N,M}(s),\\
        \A^{N,M}_k(t)&= t-(\sup\{s\in(0,t], B^{N,M}_k(s)=B^{N,M}_k(s^-)+1\}\vee0)\,,
    \end{aligned}
    \right.
\end{equation*}
with
\[ \overline{\mathfrak{F}}^{N,M}(t):=\left(\frac{1}{N}\sum_{j=1}^N\lambda_{j,B_j^{N,M}(t)}(\A^{N,M}_j(t))\right)\wedge M\,.\]
We also define
\[\overline{\mathfrak{S}}^{N,M}(t):=\frac{1}{N}\sum_{j=1}^N\gamma_{j,B^{N,M}_j(t)}(\A^{N,M}_j(t))\,.\]
We shall then prove the following theorem.

\begin{theorem}\label{LLNVIWI-M}
As $N\to\infty$, $\lbrace (\overline{\mathfrak{S}}^{N,M}(t),\overline{{\mathfrak{F}}}^{N,M}(t)), t \in [0,T] \rbrace$ converges in probability to the unique solution of 
equation \eqref{eq:limit} for the topology of uniform convergence.
\end{theorem}
We first note that Theorem \ref{LLNVIWI-M} implies Theorem \ref{LLNVIWI}. Indeed, Theorem \ref{LLNVIWI-M} implies that
\begin{equation*}
    \lim_{N\to\infty}\P\left(\sup_{0\le t\le T}\overline{{\mathfrak{F}}}^{N,M}(t)> M\right)=0\,.
\end{equation*}
But, on the event
$\left\{\sup_{0\le t\le T}\overline{{\mathfrak{F}}}^{N,M}(t)\le M\right\}$, $\overline{{\mathfrak{F}}}^{N,M}(t)=\overline{{\mathfrak{F}}}^{N}(t)$
for all $t\in[0,T]$. Hence Theorem \ref{LLNVIWI-M} implies that $(\overline{\mathfrak{S}}^{N}(t),\overline{{\mathfrak{F}}}^{N}(t))$ converges in probability as $N\to\infty$, uniformly for $t\in[0,T]$, to the unique solution of 
equation \eqref{eq:limit}.
Since $T$ was arbitrary, this yields the statement of Theorem \ref{LLNVIWI}.

Now, given Proposition \ref{prop:ident} and the law of large numbers in $\mathbf{D}$ (see \cite{Rao}), Theorem \ref{LLNVIWI-M}
will follow from the next lemma.

\begin{lemma}\label{le:diff} 
    As $N\to\infty$, the following convergences hold in probability, uniformly in $t\in[0,T]$,
    \begin{align*}
        \frac{1}{N}\sum_{k=1}^N\left[\gamma_{k,B^{N,M}_k(t)}(\A^{N,M}_k(t))-\gamma_{k,B_k(t)}(\A_k(t))\right]\to0,\\
        \frac{1}{N}\sum_{k=1}^N\left[\lambda_{k,B^{N,M}_k(t)}(\A^{N,M}_k(t))-\lambda_{k,B_k(t)}(\A_k(t))\right]\to0\,.
    \end{align*}
\end{lemma}

Lemma~\ref{le:diff} itself follows from the following lemma.
\begin{lemma}\label{le:10}
As $N\to\infty$, 
\[ \frac{1}{N}\E\left[\sum_{k=1}^N\sup_{0\le t\le T}|B^{N,M}_k(t)-B_k(t)|\right]\to0\,.\]
\end{lemma}

\begin{proof}	
We note that
\begin{multline}\label{estim10} 
\E\left[\sup_{0\le r\le t}|B^{N,M}_k(r)-B_k(r)|\right]\\
\begin{aligned}
    &\le \E\left[\int_0^t |\Upsilon_k(s)-\Upsilon_k^{N,M}(s)| ds \right] \\
    &\le\E\left[\int_0^t|\gamma_{k,B_k(s)}(\A_k(s))\overline{\mathfrak{F}}(s)-\gamma_{k,B^{N,M}_k(s)}(\A^{N,M}_k(s))\overline{\mathfrak{F}}^{N,M}(s)|ds \right]\\
    &\le M\E\left[\int_0^t|\gamma_{k,B_k(s)}(\A_k(s))-\gamma_{k,B^{N,M}_k(s)}(\A^{N,M}_k(s))|ds \right]\\
    &\qquad + \E\left[\int_0^t|\overline{\mathfrak{F}}(s)-\overline{\mathfrak{F}}^{N,M}(s)|ds\right],
\end{aligned}
\end{multline}
where we have used the fact that both $\overline{\mathfrak{F}}(t)\le M$,  $\overline{\mathfrak{F}}^{N,M}(t)\le M$, for all $0\le t\le T$ and that $ 0\leq \gamma_{k,i} \leq 1$. We first write
\begin{align}
\E\left[|\gamma_{k,B_k(s)}(\A_k(t))-\gamma_{k,B^{N,M}_k(t)}(\A^{N,M}_k(t))|\right]&\le\P\Big(\sup_{0\le r\le t}|B^{N,M}_k(r)-B_k(r)|\ge1\Big)\nonumber\\
&\le\E\Big[\sup_{0\le r\le t}|B^{N,M}_k(r)-B_k(r)|\Big]. \label{estim11}
\end{align}
In addition, we have
\begin{align}
|\overline{\mathfrak{F}}^{N,M}(t)-\overline{\mathfrak{F}}(t)|&= \left|\left(\frac{1}{N}\sum_{k=1}^N\lambda_{k,B^{N,M}_k(t)}(\A^{N,M}_k(t))\right)\wedge M-
\overline{\mathfrak{F}}(t)\wedge M\right|\nonumber\\
&\le\left|\frac{1}{N}\sum_{k=1}^N\lambda_{k,B^{N,M}_k(t)}(\A^{N,M}_k(t))-\overline{\mathfrak{F}}(t)\right|\nonumber\\
&\le\left|\frac{1}{N}\sum_{k=1}^N\left\lbrace\lambda_{k,B^{N,M}_k(t)}(\A^{N,M}_k(t))-\lambda_{k,B_k(t)}(\A_k(t))\right\rbrace\right|\nonumber\\&\quad
+\frac{1}{N}\left|\sum_{k=1}^N\left\lbrace\lambda_{k,B_k(t)}(\A_k(t))-\E\left[\lambda_{B(t)}(\A(t))\right] \right\rbrace\right|. \label{estim12}
\end{align}
We have \[ \lambda_{B(t)}(\A(t))\le\lambda^\ast_{B(t)}\le\sum_{j=0}^{B(t)}\lambda^\ast_j\le\sum_{j=0}^{\tilde{B}(t)}\lambda^\ast_j,\] where
$\tilde{B}(t):=\int_0^t\int_0^{\overline{\mathfrak{F}}(s)}Q(ds,du)$. Clearly $\tilde{B}(t)\sim Poi\left(\int_0^t\overline{\mathfrak{F}}(s)ds \right)$ and it is independent 
of the random variables $\{\lambda^\ast_j,\ j\ge0\}$. Hence
\begin{align*}
 \E\left[\lambda_{B(t)}(\A(t))\right]&\le\E[\lambda^\ast]\ \int_0^t\overline{\mathfrak{F}}(s)ds\,.
 \end{align*}
 It then follows from the usual law of large numbers for i.i.d. sequences (where the convergences holds in $L^1(\Omega)$, see
 e.g. Exercise 24 page 59 in \cite{Kal}) that the second term on the right of \eqref{estim12} tends to $0$ as $N\to\infty$, and the same is true for the integral of the same quantity on any bounded interval.

In order to complete the proof of Lemma \ref{le:10}, we define the stopping time
\[ \sigma_{N,k}:=\inf\{t\ge0,\ B^{N,M}_k(t)\not=B_k(t)\}\,,\]
and we note that $\lambda_{k,B^{N,M}_k(\sigma_{N,k}^-)}(\A^{N,M}_k(\sigma_{N,k}^-))=\lambda_{k,B_k(\sigma_{N,k}^-)}(\A_k(\sigma_{N,k}^-))=0$, since necessarily
$\gamma_{k,B^{N,M}_k(\sigma_{N,k}^-)}(\A^{N,M}_k(\sigma_{N,k}^-))=\gamma_{k,B_k(\sigma_{N,k}^-)}(\A_k(\sigma_{N,k}^-))>0$.
By \eqref{disjoint_support}, we thus have
\begin{align}\label{ineq}
&\left|\lambda_{k,B^{N,M}_k(t)}(\A^{N,M}_k(t))-\lambda_{k,B_k(t)}(\A_k(t))\right| \notag\\
&\quad\quad\quad\le{\bf1}_{\{\sigma_{N,k}\le t\}}\Big\{\lambda^\ast_{k,B^{N,M}_k(t)}{\bf1}_{B^{N,M}_k(t)>B^{N,M}_k(\sigma_{N,k}^-)}
+\lambda^\ast_{k,B_k(t)}{\bf1}_{B_k(t)>B_k(\sigma_{N,k}^-)}\Big\}.
\end{align}
%
For the second term, we write
\begin{align*}
{\bf1}_{\{\sigma_{N,k}\le t\}}\lambda^\ast_{k,B_k(t)}{\bf1}_{B_k(t)>B_k(\sigma_{N,k}^-)}
&={\bf1}_{\{\sigma_{N,k}\le t\}}\sum_{j=1}^\infty{\bf1}_{B_k(t)=j}\lambda^\ast_{k,j}{\bf1}_{\{B_k(t)> B_k(\sigma_{N,k}^-)\}}\\
&\le{\bf1}_{\{\sigma_{N,k}\le t\}}\sum_{j=1}^\infty{\bf1}_{B_k(\sigma_{N,k}^-)<j\le B_k(t)}\lambda^\ast_{k,j}\\
&={\bf1}_{\{\sigma_{N,k}\le t\}\cap\{\overline{\mathfrak{F}}(\sigma_{N,k}^-)>\overline{\mathfrak{F}}^{N,M}(\sigma_{N,k}^-)\}}\lambda^\ast_{k,B_k(\sigma_{N,k})}\\&\quad+
\int_0^t\int_0^\infty\int_{\mathbf{D}}{\bf1}_{\{\sigma_{N,k}< s\}}{\bf1}_{u\le\Upsilon_k(s^-)}\lambda^\ast \tilde{Q}_k(ds,du,d\lambda)\\
&\le{\bf1}_{\{\sigma_{N,k}\le t\}\cap\{\overline{\mathfrak{F}}(\sigma_{N,k}^-)>\overline{\mathfrak{F}}^{N,M}(\sigma_{N,k}^-)\}}\lambda^\ast_{k,B_k(\sigma_{N,k})}\\&\quad+
\int_0^t\int_0^M\int_{\mathbf{D}}{\bf1}_{\{\sigma_{N,k}< s\}}\lambda^\ast \tilde{Q}_k(ds,du,d\lambda)\,,
\end{align*}
where the PRM $\tilde{Q}_k$ is a PRM on $\R^2_+\times\mathbf{D}$ whose projection on $\R^2_+$ coincides
with $Q_k$, and its mean measure is the product of the Lebesgue measure on $\R^2_+$ with the law of 
the random function $\lambda$. Note that we have abused notations by using in the above expression $\lambda$ as the integration variable. 

We first note that $\lambda_{k,B_k(\sigma_{N,k})}$ is the $\lambda_{k,\cdot}$ choosen at the  jump of $B_k$
at time $\sigma_{N,k}$, and it is independent of all the past until that jump time. In particular it is independent
of the event $\{\sigma_{N,k}\le t\}\cap\{\overline{\mathfrak{F}}(\sigma_{N,k}^-)>\overline{\mathfrak{F}}^{N,M}(\sigma_{N,k}^-)\}$.
Hence, exploiting in addition the fact that the process ${\bf1}_{\{\sigma_{N,k}<s\}}$ is predictable,
\begin{align*}
\E\Big[{\bf1}_{\{\sigma_{N,k}\le t\}}&\lambda^\ast_{k,B_k(t)}{\bf1}_{\{B_k(t)>B_k(\sigma_{N,k}^-)\}}\Big]\\
&\le\P(\sigma_{N,k}\le t)\E\left[\lambda^\ast\right]
+\E\left[\lambda^\ast\right]\P(\sigma_{N,k}\le t) M t\\
&\le\E\left[\lambda^\ast\right]\left(1+Mt\right)\P(\sigma_{N,k}\le t)\,.
\end{align*}
Finally, thanks to our truncation, the other term on  the right hand side of \eqref{ineq} is treated exactly as what has just been done. We have proved that there exists a constant $C_T>0$ such that for any $t\in[0,T]$,
\begin{align}\label{estim13}
\E\left[\left|\lambda_{k,B^{N,M}_k(t)}(\A^{N,M}_k(t))-\lambda_{k,B_k(t)}(\A_k(t))\right|\right]
\le C_T\E\left[\sup_{0\le r\le t}|B^{N,M}_k(r)-B_k(r)|\right]\,.
\end{align}
The result follows by combining \eqref{estim10}, \eqref{estim11}, \eqref{estim12}, \eqref{estim13} and Gronwall's Lemma.

\end{proof}

\begin{remark} Clearly, the argument leading to \eqref{estim13} would be much simplified in the case where 
$\lambda^\ast$ is bounded by a constant, which was the assumption in the previous works on this topic.
\end{remark}

\begin{proof}[Proof of Lemma \ref{le:diff}]	
The first (resp. the second) statement is a consequence of \eqref{estim11} (resp. \eqref{estim13}) and Lemma \ref{le:10}.
\end{proof}


\begin{thebibliography}{99}

\bibitem{FPP1} R. Forien, G. Pang, \'E. Pardoux, Epidemic models with varying infectivity, 
{\it SIAM J. Applied Math.}, {\bf81}, pp. 1893-1930, 2021. 
\bibitem{FPP2} R. Forien, G. Pang, \'E. Pardoux, Multi-patch multi-group epidemic model with varying infectivity, {\it Probab. Uncertain. Quant. Risk} {\bf7}, 4, pp. 333-364, 2022.
\bibitem{FPPZ} R. Forien, G. Pang, \'E. Pardoux and A.B. Zotsa-Ngoufack, Stochastic epidemic models with varying infectivity and susceptibility, {\it Annals of Applied Probab.} {\bf35}, 2175--2216, 2025.
\bibitem{Kal} O. Kallenberg, {\it Foundations of Modern Probability}, Springer, 1997.
\bibitem{KK1} W.O. Kermack, A.G. McKendrick, A contribution to the mathematical theory of epidemics, {\it Proc. of the Royal Soc. London. Series A}
{\bf115}, 700-721, 1927.
\bibitem{KK2} W.O. Kermack, A.G. McKendrick, A contribution to the mathematical theory of epidemics II. The problem of endemicity  {\it Proc. of the Royal Soc. London. Series A} {\bf138}, 55-83, 1932.
\bibitem{KK3} W.O. Kermack, A.G. McKendrick, A contribution to the mathematical theory of epidemics III. Further studies of he problem of endemicity  {\it Proc. of the Royal Soc. London. Series A} {\bf141}, 94-122, 1933.
\bibitem{Rao} R.R. Rao The law of large numbers for $D[0,1]$-valued random variables.
{\it Theory of Probability \& its Applications} {\bf8}, 70--74, 1963.

\end{thebibliography}
\end{document}